\documentclass[reqno,12pt]{amsart}
\usepackage{amsmath,amsthm,amscd,amssymb,graphicx,enumerate,latexsym}

\usepackage[raiselinks,colorlinks]{hyperref}
\hypersetup{citecolor=blue}

\numberwithin{equation}{section}

\theoremstyle{plain}
\newtheorem{thm}{Theorem}[section]
\newtheorem{lemma}[thm]{Lemma}
\newtheorem{cor}[thm]{Corollary}
\newtheorem{prop}[thm]{Proposition}
\newtheorem{conj}[thm]{Conjecture}
\theoremstyle{definition}

\theoremstyle{remark}
\newtheorem{remark}{Remark}[section]

\def\Re{\mathop{\rm Re}\nolimits}
\def\Im{\mathop{\rm Im}\nolimits}

\newcommand{\ess}{\text{\rm{ess}}}
\newcommand{\ac}{\text{\rm{ac}}}

\newcommand{\s}{\text{\rm{s}}}

\DeclareMathOperator*{\tr}{tr}
\DeclareMathOperator*{\Int}{int}

\newcommand{\intt}{\text{\rm{-int}}}

\allowdisplaybreaks

\title[$\ell^2$ bounded variation]{$\ell^2$ bounded variation and absolutely continuous spectrum of Jacobi matrices}

\author[Yoram Last and Milivoje Lukic]{Yoram Last$^1$ and Milivoje Lukic$^2$}

\thanks{$^1$ Institute of Mathematics, The Hebrew University, 9190401 Jerusalem, Israel.
E-mail: ylast@math.huji.ac.il.
Supported in part by Grant No.\;2014337 from the United States-Israel Binational Science Foundation
(BSF), Jerusalem, Israel.}

\thanks{$^2$ Department of Mathematics, Rice University, Houston, TX 77005, U.S.A.
E-mail: milivoje.lukic@rice.edu.  Supported in part by NSF grant DMS-1301582.}


\begin{document}
\begin{abstract}
We disprove a conjecture of Breuer--Last--Simon~\cite{BreuerLastSimon10} concerning the absolutely continuous spectrum of Jacobi matrices with coefficients that obey an $\ell^2$ bounded variation condition with step $q$. We prove existence of a.c.\ spectrum on a smaller set than that specified by the conjecture and prove that our result is optimal.
\end{abstract}

\maketitle

\section{Introduction}

In this paper we study semi-infinite Jacobi matrices
\[
J = \begin{pmatrix}
b_1 & a_1 &  &  &  \\
a_1 & b_2 & a_2 &  &  \\
& a_2 & b_3 & a_3 &  \\
& & a_3 &  \ddots & \ddots \\
& & & \ddots & 
\end{pmatrix}
\]
where $a_n>0$, $b_n \in \mathbb{R}$. We assume that
\begin{equation}\label{1.1}
\sup_n a_n^{-1} + \sup_n  a_n + \sup_n \lvert b_n \rvert < \infty,
\end{equation}
in which case $J$ is a bounded self-adjoint operator on $\ell^2(\mathbb{N})$.
A canonical spectral measure $\mu$ corresponds to the cyclic vector $\delta_1$ through
\[
\int x^n d\mu(x) = \langle \delta_1, J^n \delta_1 \rangle, \quad n=0,1,2, \dots
\]
and if the Lebesgue decomposition of $\mu$ is 
\[
d\mu = f(x) dx + d\mu_\s,
\]
we will be interested in the essential support of the a.c.\ spectrum,
\[
\Sigma_\ac(J) = \{ x \in \mathbb{R} \mid f(x) > 0\}.
\]
This set should properly be viewed as an equivalence class of sets modulo sets of Lebesgue measure zero.
The absolutely continuous spectrum of $J$ is then equal to the essential closure of $\Sigma_\ac(J)$,
defined as the set of $x \in  \mathbb{R}$ such that $\lvert \Sigma_\ac(J) \cap (x-\epsilon, x+\epsilon) \rvert >0$ for all $\epsilon >0$; see \cite{GesztesyMakarovZinchenko08} for an expository discussion. 

For $\ell^2$ perturbations of coeficients of the free Jacobi matrix, \cite{DeiftKillip99,KillipSimon03} proved $\Sigma_\ac(J) = [-2,2]$. Their sum rule approach initiated a search for higher-order Szeg\H o theorems for Jacobi \cite{LaptevNabokoSafronov03,NazarovPeherstorferVolbergYuditskii05,Kupin04,Kupin05} and CMV matrices \cite{Verblunsky35,SimonZlatos05,GolinskiiZlatos07,Lukic7,Lukic10,GNR16,GNR17,BSZ17}, in which $\ell^2$ bounded variation conditions are combined with slow decay conditions ($\ell^p$ for some $p>2$) to prove presence of a.c.\ spectrum on the spectrum of the free case. In this paper, we consider the implications of an $\ell^2$ bounded variation condition without any decay conditions.

This paper focuses on Jacobi matrices such that for some $q \in \mathbb{N}$,
\begin{equation}
\label{1.2}
\sum_{n=1}^\infty \lvert a_{n+q} - a_n \rvert^2 +  \sum_{n=1}^\infty \lvert b_{n+q} - b_n \rvert^2 < \infty.
\end{equation}
The implications of condition \eqref{1.2} on $\Sigma_\ac(J)$ have been the subject of a series of papers and conjectures
of various levels of generality, relating the a.c.\ spectrum of $J$ to the a.c.\ spectra of its right limits. A two-sided
Jacobi matrix $J^{(r)}$ with coefficients $a_n^{(r)} > 0$, $b_n^{(r)} \in \mathbb{R}$, $n\in \mathbb{Z}$ is called a right
limit of $J$ if there is a sequence $n_j \in\mathbb{Z}$, $n_j \to +\infty$, such that for all $n\in \mathbb{Z}$,
\[
\lim_{j\to \infty} a_{n+n_j} = a^{(r)}_{n}, \qquad \lim_{j\to \infty} b_{n+n_j} = b^{(r)}_{n}.
\]
When \eqref{1.1} holds, a compactness argument shows that $J$ has at least one right limit; the same argument shows that
for every sequence $n_j \to +\infty$ there exists a subsequence which gives rise to a right limit. We will denote the set
of right limits of $J$ by $\mathcal R$. We are interested in the following conjecture from \cite{BreuerLastSimon10}.

\begin{conj}[{\cite[Conjecture 9.5]{BreuerLastSimon10}}] \label{C1.1} Let $q\in \mathbb{N}$ and let \eqref{1.2} hold. Then
\begin{equation}\label{1.3}
\Sigma_\ac(J)  = \bigcap_{\mathcal R} \sigma(J^{(r)}).
\end{equation}
\end{conj}

A narrower version of this conjecture, for $a_n\equiv 1$ and $b_n \to 0$, was previously made by Last~\cite{Last07} and proven
by Denisov~\cite{Denisov09}. Further work of Kaluzhny--Shamis~\cite{KaluzhnyShamis12} proved \eqref{1.3} in the case where the
sequences $\{a_n\}$, $\{b_n\}$ are asymptotically periodic (so there is, up to shifts, only one right limit). These results have
been carried over to orthogonal polynomials on the unit circle and extended beyond asymptotic periodicity by one of the
authors \cite{Lukic8}. Additional motivation for the conjecture is provided by work of Denisov~\cite{Denisov02} for
Schr\"odinger operators, which can be seen as a continuum analog of Corollary~\ref{C1.3} below.

However, we will construct examples which show that Conjecture~\ref{C1.1} is false for $q>1$. We will also prove a result which
establishes a.c.\ spectrum on a smaller set and our examples will show that this result is optimal. This will also imply
Conjecture~\ref{C1.1} for $q=1$.

The condition \eqref{1.2} implies
\begin{equation}\label{1.4}
\lim_{n\to \infty} \lvert a_{n+q} - a_n \rvert = \lim_{n\to \infty} \lvert b_{n+q} - b_n \rvert =  0,
\end{equation}
which implies that all right limits of $J$ are $q$-periodic, since
\[
a^{(r)}_{n+q} -a^{(r)}_n =\lim_{j\to\infty} ( a_{n+q+n_j} - a_{n+n_j}) = 0
\]
and analogously $b^{(r)}_{n+q} -b^{(r)}_n  = 0$. The discriminant of a two-sided $q$-periodic Jacobi matrix $J^{(r)}$ is defined as
\begin{equation}\label{1.5}
\Delta^{(r)}(z) =  \tr \left( A(a^{(r)}_{q},b^{(r)}_{q}; z) A( a^{(r)}_{q-1}, b^{(r)}_{q-1}; z) \dots A( a^{(r)}_1, b^{(r)}_1; z)  \right)
\end{equation}
where $\tr$ denotes trace and $A(a,b;z)$ is the transfer matrix
\begin{equation}\label{1.6}
A(a,b;z) = \begin{pmatrix} \frac{z-b}a & -\frac 1a \\ a & 0 \end{pmatrix}.
\end{equation}
It is well known \cite[Chapter 5]{Rice} that for such a Jacobi matrix,
\begin{equation}\label{1.7}
\sigma(J^{(r)}) = \sigma_\ac(J^{(r)}) = \{ x \in  \mathbb{R} \mid \Delta^{(r)}(x) \in [-2,2] \}
\end{equation}
and that this set is, in a natural way, a union of $q$ closed intervals (``bands") in $\mathbb{R}$ whose interiors are disjoint.
One can naturally define the $q$-interior of the spectrum as the union of interiors of the $q$ bands, which can be expressed as
\[
q\intt(\sigma(J^{(r)})) = \{ x \in \mathbb{R} \mid \Delta^{(r)}(x) \in (-2,2) \}.
\]
Note that this notion depends on $q$ and cannot be expressed solely in terms of $J^{(r)}$, in the sense that a $q$-periodic Jacobi
matrix can also be viewed as $2q$-periodic, $3q$-periodic, etc, and $q\intt(\sigma(J^{(r)}))$,
$(2q)\intt(\sigma(J^{(r)}))$, $(3q)\intt(\sigma(J^{(r)}))$ \dots are all distinct sets.

We can now state the theorem.

\begin{thm} \label{T1.2} Let \eqref{1.1} and \eqref{1.2} hold for some $q\in \mathbb{N}$. Then
\begin{equation}\label{1.8}
\bigcap_{\mathcal R} q\intt(\sigma(J^{(r)})) \subset \Sigma_\ac(J)  \subset \bigcap_{\mathcal R} \sigma(J^{(r)}).
\end{equation}
Moreover, for any closed interval
\begin{equation}\label{1.9}
I \subset \bigcap_{\mathcal R} q\intt(\sigma(J^{(r)})),
\end{equation}
we have
\begin{equation}\label{1.10}
\int_I  \log f(x) dx > -\infty.
\end{equation}
\end{thm}

The second inclusion in \eqref{1.8} is, in fact, a general result of Last--Simon \cite{LastSimon99} for a.c.\ spectra of right
limits, repeated here only for completeness. The essence of this theorem is in the first inclusion.

Although $q\intt(\sigma(J^{(r)}))$ differs from $\sigma(J^{(r)})$ by only a finite set of points, those points can vary from
right limit to right limit, so we would like to emphasize that the intersections in \eqref{1.8} can differ significantly.
We will soon see examples of this.

For $q=1$, the two inclusions of the previous theorem combine to give an equality.

\begin{cor}\label{C1.3} If \eqref{1.1} holds and \eqref{1.2} holds for $q=1$, then
\begin{equation}\label{1.11}
\Sigma_\ac(J) = [  \limsup_{n\to\infty} (b_n - 2 a_n),  \liminf_{n\to\infty} (b_n + 2 a_n)].
\end{equation}
Moreover, \eqref{1.10} holds for each closed interval $I \subset \Int (\Sigma_\ac(J))$.
\end{cor}

\begin{remark}
This corollary sometimes yields intervals with purely singular spectrum. A result of Last--Simon \cite[Theorem 3.1]{LastSimon06}
for essential spectra of right limits implies
\[
\sigma_\ess(J) =  [  \liminf_{n\to\infty} (b_n - 2 a_n),  \limsup_{n\to\infty} (b_n + 2 a_n)]
\]
which can be strictly greater than the set \eqref{1.11}, so the complement supports a purely singular part of the measure.
\end{remark}

Another case in which the sets in \eqref{1.8} are equal is the case of convergence to an isospectral torus.
This notion is the natural generalization of decaying perturbations of the free case; see, e.g., Last--Simon~\cite{LastSimon06}
and Damanik--Killip--Simon~\cite{DamanikKillipSimon10}.

For our purposes, it suffices to define it as follows. Let $\mathcal{S} = \sigma( \tilde J)$ for some $q$-periodic two-sided
Jacobi matrix $\tilde J$. The isospectral torus of $\mathcal{S}$, denoted $\mathcal{T}_{\mathcal{S}}$, is the set of all
$q$-periodic two-sided Jacobi matrices whose spectrum is equal to $\mathcal{S}$. It is known that this set is a $k$-dimensional
torus for some $k\le q-1$, and that all elements of the isospectral torus have the same discriminant, which we will denote by
$\Delta_{\mathcal{S}}(x)$.

We will say that $J$ converges to the isospectral torus $\mathcal T_{\mathcal{S}}$ if all of its right limits lie on
$\mathcal T_{\mathcal{S}}$. Of course, this generalizes asymptotic periodicity.

By \cite{LastSimon06}, convergence of $J$ to the isospectral torus $\mathcal{T}_{\mathcal{S}}$ implies  $\sigma_\ess(J) = \mathcal{S}$.
With our $\ell^2$ bounded variation condition \eqref{1.2}, we can also say that $\sigma_\ac(J) = \mathcal{S}$. More precisely, we have:

\begin{cor}\label{C1.4}
Let \eqref{1.1} and \eqref{1.2} hold for some $q\in\mathbb{N}$. If  $\{a_n,b_n\}_{n=1}^\infty$ converges to an isospectral
torus $\mathcal T_{\mathcal{S}}$, then
\begin{equation}\label{1.12}
\Sigma_\ac(J) = \mathcal{S}.
\end{equation}
Moreover, \eqref{1.10} holds for any closed interval $I \subset \mathcal{S}$ such that $\lvert\Delta_{\mathcal{S}}(x)\rvert < 2$ for all $x\in I$.
\end{cor}

By Corollary~\ref{C1.3}, Conjecture~\ref{C1.1} is true for $q=1$. For an arbitrary $q>1$, we will now discuss examples in which the two
intersections in \eqref{1.8} are distinct and $\Sigma_\ac(J)$ is equal to one or the other. This will show that, in general, no better
statement can be made than \eqref{1.8}.

Our examples will be taken from the class of discrete Schr\"odinger operators ($a_n\equiv 1$). Moreover, let us choose a parameter
$\lambda \in (0,2)$ and assume that the set of right limits is the set of constant Jacobi matrices with $a_n \equiv 1$ and $b_n \equiv \beta$
for $\beta \in [-\lambda, \lambda]$,
\begin{equation}\label{1.13}
\mathcal R = \{ J(1,\beta) \mid  \beta \in [-\lambda, \lambda] \}.
\end{equation}
The $q$-discriminant of the free Jacobi matrix $J(1,0)$ is
\[
\Delta(z) = \tr \left( \begin{pmatrix} z &  -1 \\ 1 & 0 \end{pmatrix}^q \right)
\]
and it is well known \cite{Rice} that
\[
\sigma(J(1,0)) = [-2 , 2]
\]
and
\[
q\intt( \sigma(J(1,0)) ) = (-2, 2) \setminus \{ z_1 ,  \dots, z_{q-1}  \}
\]
where $z_1 , \dots, z_{q-1} $ are the distinct solutions of $\Delta'(z) = 0$,
\[
z_j = 2 \cos\left( \frac{(q-j)\pi}q \right).
\]
Since $J(1,\beta)$ is just $J(1,0) + \beta$, it follows that
\[
\bigcap_{\mathcal R} \sigma(J^{(r)})  =  [-2 + \lambda, 2 - \lambda]
\]
and
\[
\bigcap_{\mathcal R} q\intt(\sigma(J^{(r)})) =  (-2 + \lambda, 2 - \lambda)  \setminus \bigcup_{j=1}^{q-1}  [z_j - \lambda, z_j + \lambda].
\]
As promised, these intersections are distinct for $\lambda \in (0,2)$.
Moreover, we may have
\[
\bigcap_{\mathcal R} q\intt(\sigma(J^{(r)})) = \emptyset
\]
even when $\bigcap_{\mathcal R} \sigma(J^{(r)})$
is a fairly large interval.
(In fact, for any positive $\lambda$, $\bigcap_{\mathcal R} q\intt(\sigma(J^{(r)}))$ is empty if $q$ is large enough.)
Now we will see that each can be the essential support of the
a.c.\ spectrum for a suitable Jacobi matrix (where empty essential support means there is no a.c.\ spectrum).
The first of these two theorems disproves Conjecture~\ref{C1.1}.

\begin{thm}\label{T1.5}
Let $q\in \mathbb{N}$, $q>1$. There exists a half-line Jacobi matrix $J$ with the properties \eqref{1.1}, \eqref{1.2} and with
$a_n \equiv 1$ such that its set of right limits is the set $\mathcal{R}$ given by \eqref{1.13} and
\[
\bigcap_{\mathcal R} q\intt(\sigma(J^{(r)})) = \Sigma_\ac(J)  \neq \bigcap_{\mathcal R} \sigma(J^{(r)}).
\]
\end{thm}

\begin{thm}\label{T1.6}
Let $q\in \mathbb{N}$, $q>1$. There exists a half-line Jacobi matrix $J$ with the properties \eqref{1.1}, \eqref{1.2} and with
$a_n \equiv 1$ such that its set of right limits is the set $\mathcal{R}$ given by \eqref{1.13} and
\[
\bigcap_{\mathcal R} q\intt(\sigma(J^{(r)})) \neq \Sigma_\ac(J)  = \bigcap_{\mathcal R} \sigma(J^{(r)}).
\]
\end{thm}

The rest of this paper is organized as follows.
In Sections~\ref{S2} and \ref{S3}, we prove Theorem~\ref{T1.2}, using the method of Denisov~\cite{Denisov09} and
Kaluzhny--Shamis~\cite{KaluzhnyShamis12} together with some adaptations first made in \cite{Lukic8} in the OPUC setting.
In Section~\ref{S4}, we apply it to Corollaries~\ref{C1.3} and \ref{C1.4}. In Section~\ref{S5} we prove Theorem~\ref{T1.5}
using a method from \cite{Last07}. In Section~\ref{S6} we prove Theorem~\ref{T1.6}.

\section{Estimates and diagonalization of $q$-step transfer matrices}\label{S2}

We denote the $q$-step transfer matrix between positions $mq$ and $(m+1)q$ and its trace and entries by
\begin{align*}
\Phi_m(z) & = A(a_{(m+1)q},b_{(m+1)q};z) A(a_{(m+1)q-1},b_{(m+1)q-1};z) \dots A(a_{mq+1},b_{mq+1};z) \\
\Delta_m(z) & = \tr \Phi_m(z) \\
\Phi_m(z) & = \begin{pmatrix} A_m(z) & B_m(z) \\  C_m(z) & D_m(z) \end{pmatrix}
\end{align*}
In this section, we prepare for the proof of Theorem~\ref{T1.2} by establishing certain properties of $\Phi_m(z)$ which will
be needed later. They are mostly uniform estimates, necessary because without asymptotic periodicity of Jacobi parameters,
we do not have convergence of $\Phi_m(z)$ in $m$. They are analogs of estimates made in \cite{Lukic8} for orthogonal polynomials
on the unit circle.

The following are standard facts about $q$-step transfer matrices \cite[Chapter 5]{Rice}.

\begin{thm} \label{T2.1}
\begin{enumerate}[{\rm (i)}]
\item $\det \Phi_m(z) = 1$;
\item $z \in  \mathbb{R}$ implies  $A_m(z), B_m(z), C_m(z), D_m(z), \Delta_m(z), \Delta_m'(z) \in \mathbb{R}$;
\item $\Delta_m(z) \in [-2,2]$ implies $z\in \mathbb{R}$;
\item $\Delta_m(z) \in (-2,2)$ implies $\Delta_m'(z)\neq 0$;
\item $\Delta_m(z) \in (-2,2)$ implies $C_m(z) \neq 0$.
\end{enumerate}
\end{thm}

Although the notation $\Phi_m(z)$ is convenient, we find it useful to think about $\Phi_m(z)$ as a fixed ($m$-independent) function of
\[
a_{mp+1},a_{mp+2},\dots,a_{(m+1)p} \in (0,\infty), \;\,
b_{mp+1},b_{mp+2},\dots,b_{(m+1)p} \in \mathbb{R}, \;\, z\in \mathbb{C}.
\]
In that point of view, note that $\Phi_m(z)$ is an analytic function of its parameters, and the same is true of
$A_m(z)$, $B_m(z)$, $C_m(z)$, $D_m(z)$ and $\Delta_m(z)$. For any such function $f_m(z)$, if \eqref{1.1} holds,
then for any compact $K\subset \mathbb{C}$, analyticity and compactness imply that there is a constant $C< \infty$
such that  for all $m\ge 0$ and $z \in K$,
\begin{align}
\lvert f_m(z) \rvert  & \le C, \label{2.1}  \\
\lvert f_{m+1}(z) - f_m(z) \rvert  & \le C \sum_{k=1}^{q} \left( \lvert a_{(m+1)q+k} - a_{mq+k} \rvert + \lvert b_{(m+1)q+k} - b_{mq+k} \rvert \right) .\label{2.2}
\end{align}

For $z\in \mathbb{C}$, let us define
\[
L(z) = \limsup_{m\to\infty} \lvert \Delta_m(z) \rvert.
\]

\begin{lemma}\label{L2.2}
Assume \eqref{1.1} and  \eqref{1.4}.  Then $L(z)$ is finite for all $z\in\mathbb{C}$,
\[
L(z) = \max_{\mathcal R} \lvert \Delta^{(r)}(z) \rvert,
\]
$L(z)$ is Lipschitz continuous on any compact subset of $\mathbb{C}$, and
\begin{equation}\label{2.3}
\bigcap_{\mathcal R} q\intt(\sigma(J^{(r)})) = \{ x\in \mathbb{R} \mid L(x) < 2\},
\end{equation}
which is an open set.
\end{lemma}

This lemma follows easily from compactness arguments and the observation that it suffices to consider right limits stemming from a
sequence of $n_j$ which are divisible by $q$. For more details, compare with Lemma~3.2 in \cite{Lukic8}.

The basic structure of the proof of Theorem~\ref{T1.2} is to pick a closed interval $I$ with the property \eqref{1.9} and
prove \eqref{1.10}. To prove \eqref{1.10}, we will need some uniform estimates which hold on such an interval. By \eqref{2.3}
and continuity of $L$,
\[
\max_{x\in I} L(x) <2.
\]

\begin{lemma}[{analogous to \cite[Lemma 3.3]{Lukic8}}] \label{L2.3}
Assume \eqref{1.1} and  \eqref{1.4}  and let $I\subset\mathbb{R}$ be a closed interval such that \eqref{1.9} holds.
Then there exist $m_0\in \mathbb{N}_0$, $s, t \in \{-1,+1\}$, $\epsilon\in (0,1)$ and $C>0$ such that for all
$m\ge m_0$ and $z\in \Omega$,
\begin{align}
\lvert \Delta_m(z) \rvert & \le 2 - C \label{2.4} \\
- s \Re \Delta'_m(z) &  \ge C \label{2.5} \\ 
C \le t \Re  C_m(z) & \le \lvert C_m(z) \rvert   \le C^{-1} \label{2.6}
\end{align}
where
\begin{equation}\label{2.7}
\Omega = \{ x+iy \mid x\in I, y \in [0,\epsilon] \}.
\end{equation}
\end{lemma}

Our next goal is to diagonalize the $\Phi_m(z)$ for $m\ge m_0$ and $z\in \Omega$ in a way which obeys certain uniform
estimates in $z$ and $m$. To do this, we choose an eigenvalue of $\Phi_m(z)$ in a consistent way. With $s$ as in \eqref{2.5},
define
\[
\lambda_{m}(z) = \frac {\Delta_m(z) \pm i s \sqrt{4-\Delta_m(z)^2}}2,
\]
where we take the branch of $\sqrt{~}$ on $\mathbb{C} \setminus (-\infty,0]$ such that $\sqrt 1 = 1$.

\begin{lemma}\label{L2.4} $\lambda_m(z)$ and $\lambda^{-1}_m(z)$ are the eigenvalues of $\Phi_m(z)$, and they obey the
following estimates for some $C>0$, uniformly in $m\ge m_0$, $z\in \Omega$:
\begin{equation}\label{2.8}
C \le - s \Im \lambda_{m}(z) \le \lvert \lambda_{m}(z) \rvert \le 1 -  C \Im z
\end{equation}
\begin{equation}\label{2.9}
s \Im \lambda^{-1}_{m}(z) \ge C.
\end{equation}
\end{lemma}

\begin{proof}
$\lambda_{m}(z)$  and $\lambda^{-1}_{m}(z)$ are eigenvalues of $\Phi_m(z)$ since $\det\Phi_m(z)=1$ and $\tr\Phi_m(z)=\Delta_m(z)$.
Note that
\[
\frac{\partial}{\partial y} \Delta_m(x+iy) =  i \Delta'_m(x+iy)
\]
so, taking imaginary parts and multiplying by $s$,
\[
s \frac{\partial}{\partial y} \Im  \Delta_m(x+i y) = s \Re  \Delta'_m(x+iy) \le - C
\]
for some $C>0$ independent of $m$ and $x+iy\in \Omega$, by \eqref{2.5}. Integrating in $y$ and using $\Im \Delta_m(x) = 0$,
\begin{equation}\label{2.10}
s \Im \Delta_m(x+iy) =  \int_0^y s \frac{\partial}{\partial y} \Im  \Delta_m( x+it) d t  \le - C y.
\end{equation}
By Lemma~4.1 of \cite{Lukic8},
\begin{equation}\label{2.11}
\lvert \lambda_m(z) \rvert \le 1 + s \Im \Delta_m(x+iy).
\end{equation}
Combining \eqref{2.10} and \eqref{2.11}, we obtain the upper bound on $\lvert \lambda_m(z)\rvert$ in \eqref{2.8}.
The bounds on $s \Im \lambda_m^{\pm 1}(z)$ follow from Lemma~4.1(iii) of \cite{Lukic8}.
\end{proof}

We now diagonalize $\Phi_m(z)$ as
\begin{equation}\label{2.12}
\Phi_m(z) = U_m(z) \Lambda_m(z) U_m(z)^{-1}
\end{equation}
where
\[
\Lambda_m(z) = \begin{pmatrix} \lambda_{m}(z) & 0 \\ 0 & \lambda^{-1}_{m}(z) \end{pmatrix}
\]
and
\begin{equation}\label{2.13}
U_m(z) = \begin{pmatrix} \lambda_{m}(z) - D_m(z) &  \lambda^{-1}_{m}(z) - D_m(z) \\  C_m(z) & C_m(z) \end{pmatrix}.
\end{equation}
We chose columns of $U_m(z)$ to be eigenvectors of $\Phi_m(z)$, ensuring \eqref{2.12}.
Note that $\det U_m = (\lambda_{m} - \lambda^{-1}_{m}) C_m \neq 0$ by \eqref{2.6} and Lemma~\ref{L2.4}.
We also compute
\begin{equation}\label{2.14}
U_m^{-1} = \frac 1{(\lambda_{m} - \lambda^{-1}_{m}) C_m} \begin{pmatrix} C_m  & D_m - \lambda^{-1}_{m} \\ - C_m & \lambda_{m} - D_m \end{pmatrix}
\end{equation}
and define
\[
W_m = U_m^{-1} U_{m+1} - I.
\]
By \eqref{2.2} and the preceding discussion, it is clear that
\[
\lVert U_{m+1} - U_m \rVert \le C \sum_{k=1}^{q} \left( \lvert a_{(m+1)q+k} - a_{mq+k} \rvert + \lvert b_{(m+1)q+k} - b_{mq+k} \rvert \right).
\]
Together with $\lVert U_m^{-1} \rVert \le C$, this implies that
\[
\lVert W_m \rVert \le C \sum_{k=1}^{q} \left( \lvert a_{(m+1)q+k} - a_{mq+k} \rvert + \lvert b_{(m+1)q+k} - b_{mq+k} \rvert \right)
\]
for some value of $C<\infty$, uniformly in $m\ge m_0$ and $z\in \Omega$, and so by \eqref{1.2},
\begin{equation}\label{2.15}
\sum_{m=0}^\infty \lVert W_m \rVert^2 < \infty.
\end{equation}

\section{Proof of Theorem~\ref{T1.2}}\label{S3}

In this section we conclude the proof of Theorem~\ref{T1.2}, adapting the method of Denisov~\cite{Denisov09} and
Kaluzhny--Shamis~\cite{KaluzhnyShamis12}.

Our first step is to follow an idea of \cite{KaluzhnyShamis12} of introducing approximants of $J$ which are eventually periodic
and relating the a.c.\ parts of their spectral measures to certain Weyl solutions. For \cite{KaluzhnyShamis12}, the coefficients
in their approximants were eventually equal to the periodic background; since we are working without asymptotic periodicity,
we instead extend by periodicity from some point on.

Therefore, we define the Jacobi matrix $J^N$, $N=0,1,\dots$, so that its first $(N+1)q$ Jacobi coefficients agree with those of $J$,
and extending the sequence of coefficients by $q$-periodicity after that; i.e., the Jacobi coefficients of $J^N$ are
\begin{align}
a^N_{mq+r} &  = a_{\min(m,N) q+r}, \quad m\in \mathbb{N}_0, \; r = 1,\dots, q \label{3.1} \\
b^N_{mq+r} & = b_{\min(m,N) q+r}, \quad m\in \mathbb{N}_0, \; r = 1,\dots, q  \label{3.2}
\end{align}
We will also use the superscript $N$ to denote other quantities corresponding to $J^N$; for instance, the $q$-step transfer matrices
corresponding to $J^N$ are, by \eqref{3.1} and \eqref{3.2},
\[
\Phi^N_m(z) = \Phi_{\min(N,m)}(z).
\]
For $z\in \Omega$ and $N\ge m_0$, we wish to single out a solution $u^N(z)$ of the transfer matrix recursion,
\[
u^N_{n+1} (z)=  \Phi^N_n (z) u^N_n(z).
\]
This is a first order recurrence relation, so since all $\Phi_n$ are invertible, we can specify the solution by setting its value at $n=N$,
\begin{equation}\label{3.3}
u^N_N(z) = \begin{pmatrix} \lambda_{N}(z) - D_N(z) \\ C_N(z) \end{pmatrix}.
\end{equation}
Let $\mu^N$, the canonical spectral measure of $J^N$, have the Lebesgue decomposition
\[
d\mu^N = f^N dx + d\mu^N_\s. 
\]
We can now describe $f^N$ in terms of $u^N$. This is a rewriting of equation (3.5) of \cite{KaluzhnyShamis12}. We deviate cosmetically
from \cite{KaluzhnyShamis12} in using a solution of the transfer matrix recursion rather than a solution of the Jacobi recursion.
We prefer this point of view because it avoids a need to extend the Jacobi recursion to the endpoint $n=0$ and because it clarifies
the analogy with the case of orthogonal polynomials on the unit circle covered in \cite{Lukic8}.

\begin{lemma} \label{L3.1}
Let $N\ge m_0$. For every $x\in I$, $(u_0^N)_2(x) \neq 0$. For Lebesgue-a.e.\ $x \in I$,
\begin{equation}\label{3.4}
f^N(x) = - \frac{ C_N(x) \Im \lambda_{N}(x) }{\pi \lvert (u^N_0)_2(x) \rvert^2}.
\end{equation}
\end{lemma}

\begin{remark}
By Theorem~\ref{T2.1}(ii), we already know that the right hand side of \eqref{3.4} is real-valued. In fact, using the above formula and
comparing $f^N(x)\ge 0$ with \eqref{2.6} and \eqref{2.8} gives $s=t$, but that observation will not be needed in what follows.
\end{remark}

\begin{proof}
For $x\in \mathbb{R}$, the matrices $\Phi_n(x)$ have real entries, so $\overline{u^N(x)}$ is also a solution of the same recursion.
By the constancy of their Wronskian (see, e.g., \cite[Prop. 3.2.3]{Rice}),
\[
 (u^N_0)_1(x) \overline{(u^N_0)_2(x) } - (u^N_0)_2(x) \overline{(u^N_0)_1(x) } = (u^N_N)_1(x) \overline{(u^N_N)_2(x) } - (u^N_N)_2(x) \overline{(u^N_N)_1(x) },
\]
which, using \eqref{3.3} and Theorem~\ref{T2.1}(ii), simplifies to
\begin{equation}\label{3.5}
 \Im  ( (u^N_0)_1(x) \overline{(u^N_0)_2(x) } )  = \Im ( (u^N_N)_1(x) \overline{(u^N_N)_2(x) } )
  =  C_N(x) \Im \lambda_{N}(x).
\end{equation}
In particular, by \eqref{2.6} and \eqref{2.8}, this implies that $(u^N_0)_1(x) \overline{ (u^N_0)_2(x)} \neq 0$ for $x\in I$.

For $z\in \Omega \setminus I$, from  $\Phi^N_n u^N_N = \lambda_{N} u^N_N$ for $n\ge N$ and $\lvert \lambda_{N} \rvert < 1$ it follows
that $u^N_n$ is a Weyl solution (see, e.g., \cite[Section 3.2]{Rice}). Thus, $u_0^N(z)$ is a multiple of
$\begin{pmatrix} m^N(z) \\ -1 \end{pmatrix}$, where $m^N$ is the Weyl $m$-function for $J^N$. Thus,
\[
m(z) = - \frac{(u^N_0)_1(z)}{(u^N_0)_2(z)}.
\]
For almost every $x\in \mathbb{R}$, the nontangential limit of $\Im m^N(x)$ is equal to $\pi f^N(x)$, so
\[
f^N(x) = - \frac 1\pi \lim_{\epsilon \downarrow 0} \Im \frac{(u^N_0)_1(x+i\epsilon)}{(u^N_0)_2(x+i\epsilon )}.
\]
The limit exists for all $x \in I$ because $u^N_N$, and so $u^N_n$ for every $n$, is continuous in $z \in \Omega$.
Using \eqref{3.5}, this simplifies to \eqref{3.4}.
\end{proof}

Coefficient stripping is the operation of removing the leading Jacobi coefficients from the Jacobi matrix, i.e.\ replacing
the sequence of coefficients $\{a_n,b_n\}_{n=1}^\infty$ by $\{a_n,b_n\}_{n=2}^\infty$. This operation does not affect the
validity of conclusions of Theorem~\ref{T1.2}, so we perform coefficient stripping finitely many times and prove the result
for the Jacobi matrix obtained in this way, from which the result for the original Jacobi matrix will follow. 

Thus, in the following we may assume that all the above estimates, derived for $m\ge m_0$, now hold for all $m\ge 0$,
and that, instead of \eqref{2.15},
\begin{equation}\label{3.6}
\sum_{n=0}^\infty \lVert W_n\rVert^2 < \delta
\end{equation}
for a suitably chosen $\delta >0$.

The recursion relation for $u^N_n$, solved backwards, gives
\[
u^N_0 = \tilde \Phi_0^{-1}  \cdots \tilde \Phi_{N-1}^{-1} u^N_N.
\]
Using the diagonalization of $\tilde\Phi_n$ and computing $U_N^{-1} u^N_N = \begin{pmatrix} 1\\ 0 \end{pmatrix}$, this becomes
\begin{equation}\label{3.7}
U_0^{-1} u^N_0 =   \Lambda_0^{-1} (I + W_0) \cdots \Lambda_{N-1}^{-1} (I + W_{N-1}) \begin{pmatrix} 1\\ 0 \end{pmatrix}.
\end{equation}
Let us label the entries of $W_n(z)$,
\[
W_n (z) = \begin{pmatrix} E_n(z) & F_n(z) \\  G_n(z) & H_n(z) \end{pmatrix}.
\]
From \eqref{2.13} and \eqref{2.14} we compute
\begin{align*}
1+ E_n & =  \frac {C_n(\lambda_{n+1} - D_{n+1}) + C_{n+1} (D_n - \lambda^{-1}_{n})}{(\lambda_{n} - \lambda^{-1}_{n}) C_n}, \\
1+ H_n & = \frac {-C_n(\lambda^{-1}_{n+1} - D_{n+1}) + C_{n+1}(\lambda_{n} - D_n)} {(\lambda_{n} - \lambda^{-1}_{n}) C_n}.
\end{align*}
We will need the inequalities
\begin{equation}\label{3.8}
\left\lvert \log \prod_{n=k}^l \lvert 1 + E_n \rvert \right\rvert \le C + C v \sqrt{l-k}
\end{equation}
\begin{equation}\label{3.9}
\left\lvert \log \prod_{n=k}^l \lvert 1 + H_n \rvert \right\rvert \le C + C v \sqrt{l-k}
\end{equation}
with $v = \Im z$ and with a constant $C$ independent of $z \in \Omega$. This is proved almost as in the proof of Theorem 2.2
of \cite{Denisov09}; a modification is needed where \cite{Denisov09} uses convergence of coefficients, so Lemma~2.5 of
\cite{Denisov09} must be replaced by Lemma 6.2 of \cite{Lukic8}.

We now have all the estimates needed to apply a theorem of Denisov~\cite{Denisov09}, made precisely to estimate such expressions.

\begin{thm}[{\cite[Theorem 2.1]{Denisov09}}]
Assume that \eqref{3.7} holds, that
\begin{equation}
C > \lvert \lambda_n^{-1} \rvert >  \kappa >1
\end{equation}
and that \eqref{3.6} for a sufficiently small $\delta$. Assume also there is a constant $v \in [0,1)$ such that \eqref{3.8},
\eqref{3.9} hold. Then there is a value of $C_1 \in (0,\infty)$, which depends only on $C$, such that
\begin{equation}\label{3.11}
U_0^{-1} u^N_0 = \prod_{j=0}^{N-1} \left( \lambda_j^{-1} (1+E_j) \right) \begin{pmatrix} \phi_N \\ \nu_N \end{pmatrix}
\end{equation}
where
\begin{equation}\label{3.12}
\lvert \phi_N \rvert, \lvert \nu_N \rvert \le C_1 \exp \left( \frac {C_1}{ \kappa -1} \exp\left( \frac{C_1 v^2}{\kappa - 1 } \right)  \right)
\end{equation}
Moreover, for any fixed $\epsilon>0$ and $\kappa> 1+\epsilon$, we have
\begin{equation}\label{3.13}
\lvert \phi_N \rvert > C_1^{-1} > 0, \qquad \lvert \nu_N \rvert < C_1 \sum_{j=0}^\infty \lVert W_j \rVert^2
\end{equation}
uniformly in $N$.
\end{thm}

By \eqref{2.8}, this theorem is applicable to our case, with $\kappa = 1 + C \Im z$ and $v =  \Im z $. and we conclude that
\eqref{3.11} holds. By \eqref{3.12} and since $v^2/(\kappa-1) = \Im z / C$ is uniformly bounded
for $z\in \Omega$,  $\phi_N$, $\nu_N$ obey
\begin{equation}\label{3.14}
\lvert \phi_N \rvert, \lvert \nu_N \rvert \le \exp\left( \frac{C}{\Im z} \right) 
\end{equation}
for some $C<\infty$ and all $N$ and $z\in \Omega$. Moreover, if $\delta$ in \eqref{3.6} has been chosen small enough, then by \eqref{3.13},
\begin{equation}\label{3.15}
\lvert \phi_N \rvert > C, \qquad \lvert \nu_N \rvert < \frac C2, \qquad \text{for $z\in\Omega$ with $ \Im z > \frac \epsilon 2$}.
\end{equation}
Multiplying \eqref{3.11} by $U_0(z)$ and using \eqref{2.13}, we see
\begin{equation}\label{3.16}
(u_0^N)_2(z) = \prod_{n=1}^N \left(\lambda_n^{-1}(z) (1+E_n(z))\right) C_0(z) (\phi_N(z) + \nu_N(z) )
\end{equation}
which we rewrite as
\begin{equation}\label{3.17}
-\log \lvert (u_0^N)_2(z)\rvert = - \log \prod_{n=1}^N  \left \lvert \lambda_n^{-1}(z) (1+E_n(z)) \right\rvert - \log \lvert C_0(z) \rvert  + g_N(z)
\end{equation}
where
\[
g_N(z) =  - \log \left\lvert \phi_N(z) + \nu_N(z) \right\rvert.
\]
The above estimates imply the following lemma (the proof is analogous to the proof of Lemma~6.3 of \cite{Lukic8}).

\begin{lemma} \label{L3.3}
The function $g_N(z)$ is continuous on $\Omega$ and harmonic on $\Int \Omega$. There is a value of $C\in (0,\infty)$,
independent of $N\in \mathbb{N}_0$, such that
\begin{enumerate}[{\rm (i)}]
\item for all $x\in I$ and $N\in \mathbb{N}_0$,
\begin{equation}\label{3.18}
\left\lvert  \log  f^N(x)   - 2 g_N(x)  \right\rvert  \le C
\end{equation}
\item for all $N\in \mathbb{N}_0$,
\begin{equation}\label{3.19}
\int_I g_N^+(x) dx  \le C
\end{equation}
\item for all $z\in \Omega\setminus I$ and $N\in \mathbb{N}_0$,
\begin{equation}\label{3.20}
g_N(z) \ge - \frac C{ \Im z}
\end{equation}
\item for all $z\in \Omega$ with $ \Im z > \tfrac 12 \epsilon$ and $N\in \mathbb{N}_0$,
\begin{equation}\label{3.21}
g_N(z) \le C.
\end{equation}
\end{enumerate}
\end{lemma}

We will also need the following lemma.

\begin{lemma}[{\cite{Denisov09}, \cite[Lemma 2]{KaluzhnyShamis12}}] \label{L3.4}
Assume that $f(z)$ is continuous on $\Omega$, harmonic on $\Int \Omega$, and for some $C, \alpha >0$,
\[
\int_I g^+(x) dx < C,
\]
$g(x+iy) > - C y^{-\beta}$ for $x+iy\in \Int\Omega$, and $g(x+iy)<C$ for $x+iy \in\Omega$ with $y > \frac{C}{1+\beta}$.
Then there is a constant $B$, depending only on $C,\beta$, so that
\[
\int_I g^-(x) dx < B.
\]
\end{lemma}

By Lemma~\ref{L3.3}, Lemma~\ref{L3.4} is applicable to $g_N(z)$, and proves
\[
\int_I g_N(x) dx > C
\]
with a constant $C$ independent of $N$. By \eqref{3.19} and \eqref{3.18}, this implies
\[
\int_I \log f^{N}(x) dx > C
\]
with a constant $C$ independent of $N$.

This integral is a relative entropy. Since $J^N$ converge strongly to $J$, the measures $\mu^N$ converge weakly to $\mu$,
so by upper semicontinuity of entropy \cite[Theorem 2.2.3]{Rice},
\[
\int_I \log f(x) dx \ge \limsup_{N\to\infty}  \int_I \log f^{N}(x) dx \ge C > -\infty
\]
which proves \eqref{1.10}. Thus, $\log f(x) > -\infty$, and thus $f(x) > 0$, for a.e. $x \in I$.

Note that by \eqref{2.5}, for any $x$ in the set
\[
S = \bigcap_{\mathcal R} q\intt(\sigma(J^{(r)})),
\]
all right limits have the same sign of $(\Delta^{(r)})'(x)$. Let $J$ be a band in the spectrum of some right limit of $J$.
Then $(\Delta^{(r)})'(x)$ have constant sign for all right limits and all $x\in J\cap S$, so $J\cap S$ is an interval or the
empty set. Since this is true for any of the $q$ bands, we see that $S$ is the union of at most $q$ open intervals.

Thus, $S$ can be written as a countable union of closed intervals $I$. By the above, for each such $I$, $\{x \in I \mid f(x) = 0\}$
has zero Lebesgue measure, so we conclude that $\{ x \in S \mid f(x) = 0 \}$ has zero Lebesgue measure and the first inclusion
of \eqref{1.8} follows. The second inclusion of \eqref{1.8} is a general result of Last--Simon \cite{LastSimon99}, which completes
the proof of Theorem~\ref{T1.2}.

\section{Proofs of Corollaries~\ref{C1.3} and \ref{C1.4}}\label{S4}

\begin{proof}[Proof of Corollary~\ref{C1.3}]
In this case all right limits are $1$-periodic, with $a_n^{(r)} = \alpha^{(r)}$ and $b_n^{(r)} = \beta^{(r)}$ for some
$\alpha^{(r)} > 0$ and $\beta^{(r)} \in \mathbb{R}$. For such a right limit, by \eqref{1.5} and \eqref{1.7},
\[
\sigma(J^{(r)}) = [ \beta^{(r)} - 2 \alpha^{(r)}, \beta^{(r)} + 2 \alpha^{(r)} ]
\]
 and
 \[
 1\intt(\sigma(J^{(r)}))= (\beta^{(r)} - 2 \alpha^{(r)}, \beta^{(r)} + 2 \alpha^{(r)} ).
  \]
 Since every sequence of $n_j$ has a subsequence which gives rise to a right limit,
denoting $A_\pm = \pm \liminf_{n\to\infty} (2 a_n \pm b_n)$, \eqref{1.8} becomes
\[
(-A_-, A_+) \subset \Sigma_\ac(J) \subset [-A_-, A_+].
\]
The difference between the left and right hand sides is a finite set, which is negligible since $\Sigma_\ac(J)$
is only defined up to a set of Lebesgue measure zero, so \eqref{1.11} follows.
\end{proof}

\begin{proof}[Proof of Corollary~\ref{C1.4}]
Since all right limits are $q$-periodic and have the same spectrum $\mathcal{S}$, they have the same discriminant
$\Delta_{\mathcal{S}}(x)$ (see, e.g., \cite[Chapter 5]{Rice}), so \eqref{1.8} becomes
\[
\{ x\in \mathbb{R} \mid \Delta_{\mathcal{S}}(x) \in (-2,2) \} \subset \Sigma_\ac(J) \subset \{ x\in \mathbb{R} \mid \Delta_{\mathcal{S}}(x) \in [-2,2] \}.
\]
Since $\Delta_{\mathcal{S}}$ is a nontrivial polynomial, the difference between the left and right hand sides is
a finite set, and \eqref{1.12} follows as in the previous proof.
\end{proof}

\section{Proof of Theorem~\ref{T1.5}}\label{S5}

To prove this theorem, we will rely on a method from \cite{Last07}. The sequence $b_n$ will be constructed out of two parts,
\[
b_n = \lambda_n + W_n,
\]
where $\lambda_n$ will be a piecewise constant sequence which will oscillate between $-\lambda$ and $\lambda$,
and $W_n$ will be a product of a piecewise constant decaying sequence and a periodic sequence.

To construct $W_n$, we will pick a sequence of integers
\begin{equation}\label{5.1}
0 = L_1 < L_2 < \dots,
\end{equation}
a $q$-periodic sequence $\{V_n\}_{n=1}^\infty$ with
\[
V_1 = V_2 = \dots = V_{q-1} = 0, \quad V_q = 1
\]
and a decaying sequence $w_j$ with $w_1 \le 1$,
\[
w_1 > w_2 > \dots \to 0.
\]
Then we choose
\[
W_n = w_l V_n, \qquad L_l < n \le L_{l+1}.
\]
Note that this makes $\{W_n\}$ $q$-periodic between $L_l$ and $L_{l+1}$. It is immediate that
\[
\lim_{n\to \infty} W_n = 0
\]
and
\[
\sum_{n=1}^\infty \lvert W_{n+q} - W_n \rvert^2 \le q \sum_{l=1}^\infty \lvert w_{l+1} - w_l \rvert^2  < \infty.
\]
To construct the sequence $\lambda_n$, we will refine the partition \eqref{5.1} by choosing a sequence $\{m_l\}_{l=1}^\infty$ such that
\[
m_l  \ge 2^l
\]
and, for each $l \in \mathbb{N}$, a sequence of integers
\[
L_l = n_{l, 0}  < n_{l,1} < \dots < n_{l,m_l} = L_{l+1}.
\]
Then we will pick $\lambda_n$ to be constant between $n_{l, k}$ and $n_{l,k+1}$,
\[
\lambda_n = (-1)^l \left( 1 - \frac{2k}{m_l} \right) \lambda, \qquad n_{l,k} <  n \le  n_{l,k+1}
\]
It is then straightforward to check that
\[
\limsup_{n\to\infty} \lambda_n = \lambda, \qquad \liminf_{n\to\infty} \lambda_n = - \lambda,
\]
and
\[
\sum_{n=1}^\infty \lvert \lambda_{n+1} - \lambda_n \rvert^2 \le  4 \lambda^2 \sum_{l=1}^\infty \frac 1{m_l} \le 4 \lambda^2 \sum_{l=1}^\infty \frac 1{2^l} < \infty.
\]

It follows from the above that such a Jacobi matrix has the properties \eqref{1.1}, \eqref{1.2} and the correct set of right limits,
so Theorem~\ref{T1.2} implies that
\begin{equation}\label{5.2}
 (-2 + \lambda, 2 - \lambda)  \setminus \bigcup_{j=1}^{q-1}  [z_j - \lambda, z_j + \lambda]  \subset \Sigma_\ac(J) \subset  [-2 + \lambda, 2 - \lambda].
\end{equation}
Therefore, to prove Theorem~\ref{T1.5},  it suffices to show that we can choose the parameters $\{L_l\}$, $\{m_l\}$ and $\{n_{l,k}\}$
consistently with the above constraints, in such a way that there is no a.c.\ spectrum on $(z_j - \lambda, z_j + \lambda)$, for $j=1,\dots, q-1$.

This will be accomplished with the help of the following two propositions from \cite{Last07}, which, as pointed out there,
follow from \cite{LastSimon99}. We denote by $T_{m,n}(x)$ the transfer matrix from $m$ to $n$, i.e.
\[
T_{m,n}(x) = A(a_n,b_n;x) A(a_{n-1}, b_{n-1};x) \dots A(a_m,b_m;x).
\]

\begin{prop}\label{P6.1}
For a.e.\ $x \in \Sigma_\ac(J)$,
\begin{equation}\label{5.3}
\limsup_{N \to\infty} \frac 1{N \log^2 N}  \sum_{n=1}^N \lVert T_{1,n}(x) \rVert^2 < \infty.
\end{equation}
\end{prop}

\begin{prop}\label{P6.2}
Let $J$, $\tilde J$ be discrete Schr\"odinger operators on $\ell^2(\mathbb{N})$. Suppose that for some
$m, k\in \mathbb{N}$, $k>4$, we have $b_n = \tilde b_n$ for $n \in \{ m, m+1, \dots, k\}$, and that for
some $E\in \mathbb{R}$ and $\delta>0$, $\sigma(\tilde J) \cap (E -\delta, E+\delta) = \emptyset$.
Then for $l \in \{ 4, 5, \dots, k\}$,
\[
\lVert T_{m, m+l}(E)  \rVert \ge \frac 12 \delta^2 (1+\delta^2)^{\frac{l-3}2}.
\]
\end{prop}

Let us also note an obvious crude estimate which we will need later. Our Jacobi matrix has $a_n = 1$
and $\lvert b_n\rvert \le \lambda + w_1 \le 3$ for all $n$, so for $x \in \sigma(J) \subset [-5, 5]$,
\[
\lVert A(1,b;x) \rVert \le 2 + \lvert x - b \rvert \le 10
\]
which implies
\[
\lVert T_{1,n}(x) \rVert \le 10^n.
\]

The idea of this construction is to have a sequence $b_n$ which locally looks like a constant $\lambda_n$
plus the periodic potential $V_n$ with coupling constant $w_l$. As we slowly modulate $\lambda_n$, we will
slowly move the gaps of $w_l V_n$, covering intervals of approximate length $2\lambda$. By keeping a gap
over a point $x$ long enough (i.e.\ by making $n_{l,k+1} - n_{l,k}$ long enough), we will be able to use
Proposition~\ref{P6.2} to show increase of norms of transfer matrices at $x$, which will contradict \eqref{5.3}
and show absence of a.c.\ spectrum at $x$.

Therefore, the only property we need about the potential $V_n$ is that for any positive value of the coupling
constant, all gaps are open. For any $w>0$, let us consider the $q$-periodic discrete Schr\"odinger operator
$J_w$ with diagonal terms $w V_n$.

\begin{lemma}
For any $w>0$, the discrete $q$-periodic Schr\"odinger operator $J_w$ with potential $w V_n$ has $q-1$ open gaps.
\end{lemma}

\begin{proof}
$J_w$ is a $q$-periodic discrete Schr\"odinger operator, and its $q$-step transfer matrix is
\begin{align*}
\Phi(x) & = 
\begin{pmatrix}
x - w & -1 \\
1 & 0
\end{pmatrix}
\begin{pmatrix}
x & -1 \\
1 & 0
\end{pmatrix}^{q-1} \\
& = \begin{pmatrix}
p_q(x) - w p_{q-1}(x) & - p_{q-1}(x) + w p_{q-2}(x) \\
p_{q-1}(x) & -p_{q-2}(x)
\end{pmatrix}
\end{align*}
where $p_n$ are Chebyshev polynomials of the second kind, given by $p_n(2\cos \theta) = \frac{\sin((n+1)\theta)}{\sin \theta}$.
A closed gap at $x$ would imply that $\Phi(x) = \pm I$, which would imply $p_{q-1}(x) = p_{q-2}(x) =0$. These polynomials obey
the recurrence relation $p_{n+1}(x) = x p_n(x) - p_{n-1}(x)$; using the recurrence relation backwards, this would imply
$p_n(x) = 0$ for $n=q-3, q-4, \dots, 0$, which would contradict $p_0(x) = 1$.
\end{proof}

For $j=1,\dots, q-1$, denote by $z_{l,j}$ the center of the $j$-th gap of $J_{w_l}$. Let $\delta_l$ denote the minimum width of
a gap of $J_{w_l}$ and pick $m_l$ so that $m_l \ge  4 / \delta_l$.

Then $\lambda_n$ oscillates from $-\lambda$ to $\lambda$ with steps of size $\frac{2}{m_l} \le \frac{\delta_l}2$, so for every
\begin{equation}\label{5.4}
x \in \left [ z_{l,j} - \lambda + \frac{3\delta_l}4, z_{l,j} + \lambda -  \frac{3\delta_l}4 \right],
\end{equation}
there is a value of $k  \in  \{  0, \dots, m_l -1 \}$ such that
\[
\left\lvert x - z_{l,j} - (-1)^l \left( 1 - \frac{2k}{m_l} \right) \lambda \right\rvert \le \frac{\delta_l}4,
\]
i.e.
\[
\left( x - \frac{\delta_l}4, x - \frac{\delta_l}4 \right) \cap  \sigma\left(J_{w_l} +  (-1)^l \left( 1 - \frac{2k}{m_l} \right)  \lambda \right) = \emptyset
\]
Since $J$ coincides with $J_{w_l}  +  (-1)^l \left( 1 - \frac{2k}{m_l } \right)  \lambda$ at positions
$n_{l,k}+1, \dots, n_{l,k+1}$, we can apply Proposition~\ref{P6.2} to conclude
\begin{align*}
\sum_{n=1}^{n_{l,k+1}} \lVert T_{1,n}(x)\rVert^2 & \ge \sum_{n=n_{l,k}+5}^{n_{l,k+1}} \lVert T_{1,n_{l,k}} \rVert^{-2} \lVert T_{n_{l,k}+1,n}(x)\rVert^2 \\
& \ge 10^{- 2 n_{l,k}} \sum_{n=n_{l,k}+5}^{n_{l,k+1}}  \lVert T_{n_{l,k}+1,n}(x)\rVert^2  \\
& \ge 10^{- 2 n_{l,k}} \sum_{n=n_{l,k}+5}^{n_{l,k+1}}  \frac 14 \left( \frac{\delta_l}4 \right)^4 \left( 1+ \left(  \frac{\delta_l}4 \right)^2 \right)^{n-n_{l,k}-4}
\end{align*}
The right hand side grows exponentially as a function of $n_{l,k+1}$, so we can pick $n_{l,k+1}$ sufficiently large so that the
right hand side is larger than $l n_{l,k+1} \log^2 n_{l,k+1}$. This will accomplish
\[
\frac 1{n_{l,k+1} \log^2 n_{l,k+1} } \sum_{n=1}^{n_{l,k+1}} \lVert T_{1,n}(x)\rVert^2 \ge l.
\]
If we construct the $n_{l,k}$ inductively in this way, starting from $n_{l,0} = L_l$ and stopping at $n_{l,m_l}$, we will have for
every $x$ from \eqref{5.4}, 
\begin{equation}\label{5.5}
\sup_{N \le L_{l+1} } \frac 1{N \log^2 N}  \sum_{n=1}^N \lVert T_{1,n}(x) \rVert^2  \ge l.
\end{equation}
As $l \to \infty$, $w_l \to 0$, so $\delta_l \to 0$ and $z_{l,j} \to z_j$ for $j=1,\dots, q-1$.
Thus, for any $x \in (z_j - \lambda, z_j + \lambda)$, \eqref{5.4} holds for sufficiently large $l$. Therefore, \eqref{5.5}
also holds for large enough $l$, implying
\[
\limsup_{N \to\infty} \frac 1{N \log^2 N}  \sum_{n=1}^N \lVert T_{1,n}(x) \rVert^2 = \infty.
\]
By Proposition~\ref{P6.1}, this implies that there is no a.c.\ spectrum on $(z_j - \lambda, z_j+ \lambda)$ for $j=1, \dots, q-1$.
Combined with \eqref{5.2}, this implies
\[
\Sigma_\ac(J) =  (-2 + \lambda, 2 - \lambda)  \setminus \bigcup_{j=1}^{q-1}  [z_j - \lambda, z_j + \lambda],
\]
which completes the proof.

\section{Proof of Theorem~\ref{T1.6}} \label{S6}

To construct a Jacobi matrix with the desired properties, it suffices to take $a_n \equiv 1$ and pick a sequence $b_n$ such that
\begin{align*}
\limsup_{n\to \infty} b_n & = \lambda \\
\liminf_{n\to \infty} b_n & = - \lambda
\end{align*}
and
\[
\sum_{n=1}^\infty \lvert b_{n+1} - b_n \rvert^2 < \infty.
\]
For instance, we may choose
\begin{equation}\label{6.1}
b_n = \lambda \cos (n^\gamma)
\end{equation}
for $\gamma \in (0,\tfrac 12)$. This clearly obeys \eqref{1.1} and \eqref{1.2} for the given $q$. But by Corollary~\ref{C1.3},
\begin{equation}\label{6.2}
\Sigma_\ac(J) = [ - 2 + \lambda,  2 - \lambda ],
\end{equation}
which completes the proof.

We should remark here that the Jacobi matrix given by \eqref{6.1} is also in a class of slowly oscillating Jacobi matrices
studied by Stolz~\cite{Stolz94}, who proved \eqref{6.2} by different methods.

\bibliographystyle{amsplain}

\begin{thebibliography}{10}

\bibitem{BreuerLastSimon10}
Breuer, J., Last, Y., Simon, B., \emph{The {N}evai condition},
  Constr. Approx. \textbf{32} (2010), no.~2, 221--254.

\bibitem{BSZ17}
Breuer, J., Simon, B., Zeitouni, O., \emph{Large deviations and sum rules for spectral theory: a pedagogical approach}, J. Spectr. Theory, to appear. arXiv:1608.01467.

\bibitem{DamanikKillipSimon10}
Damanik, D., Killip, R., Simon, B., \emph{Perturbations of orthogonal
  polynomials with periodic recursion coefficients}, Ann. of Math. (2)
  \textbf{171} (2010), no.~3, 1931--2010.
  
\bibitem{DeiftKillip99}
Deift, P., Killip, R., \emph{On the absolutely continuous spectrum of
  one-dimensional {S}chr\"odinger operators with square summable potentials},
  Comm. Math. Phys. \textbf{203} (1999), no.~2, 341--347.

\bibitem{Denisov02}
Denisov, S., \emph{On the existence of the absolutely continuous
  component for the measure associated with some orthogonal systems}, Comm.
  Math. Phys. \textbf{226} (2002), no.~1, 205--220. 

\bibitem{Denisov09}
Denisov, S., \emph{On a conjecture by {Y}. {L}ast}, J. Approx. Theory \textbf{158}
  (2009), no.~2, 194--213. 

\bibitem{GNR16}
Gamboa, F., Nagel, J., Rouault, A., \emph{Sum rules via large deviations}, J.\
Funct.\ Anal.\ \textbf{270} (2016), 509--559.
  
\bibitem{GNR17}
Gamboa, F., Nagel, J., Rouault, A., \emph{Sum rules and large deviations for spectral measures on the unit circle}, preprint.  arXiv:1601.08135.

\bibitem{GesztesyMakarovZinchenko08}
Gesztesy, F., Makarov, K., Zinchenko, M., \emph{Essential
  closures and {AC} spectra for reflectionless {CMV}, {J}acobi, and
  {S}chr\"odinger operators revisited}, Acta Appl. Math. \textbf{103} (2008),
  no.~3, 315--339. 

\bibitem{GolinskiiZlatos07}
Golinskii, L., Zlato{\v{s}}, A., \emph{Coefficients of orthogonal
  polynomials on the unit circle and higher-order {S}zeg{\H o} theorems},
  Constr. Approx. \textbf{26} (2007), no.~3, 361--382. 

\bibitem{KaluzhnyShamis12}
Kaluzhny, U., Shamis, M., \emph{Preservation of absolutely continuous spectrum
  of periodic {J}acobi operators under perturbations of square-summable
  variation}, Constr. Approx. \textbf{35} (2012), no.~1, 89--105. 

\bibitem{KillipSimon03}
Killip, R., Simon, B., \emph{Sum rules for {J}acobi matrices and their
  applications to spectral theory}, Ann. of Math. (2) \textbf{158} (2003),
  no.~1, 253--321.
  
\bibitem{Kupin04}
Kupin, S., \emph{On a spectral property of {J}acobi matrices}, Proc. Amer. Math.
  Soc. \textbf{132} (2004), no.~5, 1377--1383.
  
\bibitem{Kupin05}
Kupin, S., \emph{Spectral properties of {J}acobi matrices and sum rules of
  special form}, J. Funct. Anal. \textbf{227} (2005), no.~1, 1--29.

\bibitem{LaptevNabokoSafronov03}
Laptev, A., Naboko, S., Safronov, O., \emph{On new relations between spectral
  properties of {J}acobi matrices and their coefficients}, Comm. Math. Phys.
  \textbf{241} (2003), no.~1, 91--110.

\bibitem{Last07}
Last, Y., \emph{Destruction of absolutely continuous spectrum by perturbation
  potentials of bounded variation}, Comm. Math. Phys. \textbf{274} (2007),
  no.~1, 243--252.

\bibitem{LastSimon99}
Last, Y., Simon, B., \emph{Eigenfunctions, transfer matrices, and
  absolutely continuous spectrum of one-dimensional {S}chr\"odinger operators},
  Invent. Math. \textbf{135} (1999), no.~2, 329--367. 
  
\bibitem{LastSimon06}
Last, Y., Simon, B., \emph{The essential spectrum of {S}chr\"odinger, {J}acobi, and {CMV}
  operators}, J. Anal. Math. \textbf{98} (2006), 183--220.


\bibitem{Lukic7}
Lukic, M., \emph{On a conjecture for higher-order {S}zeg{\H o} theorems}, Constr.
  Approx. \textbf{38} (2013), 161--169.
  
\bibitem{Lukic8}
Lukic, M., \emph{Square-summable variation and absolutely continuous spectrum},
J. Spectr. Theory \textbf{4} (2014), no.~4, 815–--840.
  
\bibitem{Lukic10}
Lukic, M.,  \emph{On higher-order {S}zeg{\H o} theorems with a single critical point of arbitrary order}, Constr.  Approx. \textbf{44} (2016), 283--296.
  
 \bibitem{MolchanovNovitskiiVainberg01}
Molchanov, S., Novitskii, M., Vainberg, B., \emph{First {K}d{V} integrals and
  absolutely continuous spectrum for 1-{D} {S}chr\"odinger operator}, Comm.
  Math. Phys. \textbf{216} (2001), no.~1, 195--213.

\bibitem{NazarovPeherstorferVolbergYuditskii05}
Nazarov, F., Peherstorfer, F., Volberg, A., Yuditskii, P., \emph{On generalized
  sum rules for {J}acobi matrices}, Int. Math. Res. Not. (2005), no.~3,
  155--186.

\bibitem{Rice}
Simon, B., \emph{Szeg{\H o}'s theorem and its descendants. {S}pectral theory
  for ${L{^{2}}}$ perturbations of orthogonal polynomials}, M. B. Porter
  Lectures, Princeton University Press, Princeton, NJ, 2011.

\bibitem{SimonZlatos05}
Simon, B., Zlato{\v{s}}, A.,, \emph{Higher-order {S}zeg{\H o} theorems
  with two singular points}, J. Approx. Theory \textbf{134} (2005), no.~1,
  114--129.

\bibitem{Stolz94}
Stolz, G., \emph{Spectral theory for slowly oscillating potentials. {I}.
  {J}acobi matrices}, Manuscripta Math. \textbf{84} (1994), no.~3-4, 245--260.
 
  \bibitem{Verblunsky35}
Verblunsky, S., \emph{On {P}ositive {H}armonic {F}unctions: {A} {C}ontribution
  to the {A}lgebra of {F}ourier {S}eries}, Proc. London Math. Soc.
  \textbf{S2-38} (1935), no.~1, 125--157.


\end{thebibliography}

\providecommand{\bysame}{\leavevmode\hbox to3em{\hrulefill}\thinspace}
\providecommand{\MR}{\relax\ifhmode\unskip\space\fi MR }
\providecommand{\MRhref}[2]{%
  \href{http://www.ams.org/mathscinet-getitem?mr=#1}{#2}
}
\providecommand{\href}[2]{#2}

\end{document}